\RequirePackage{fix-cm}
\documentclass[smallextended]{svjour3}
\smartqed  
\usepackage{graphicx}
\usepackage{epsfig,amsfonts,amssymb,latexsym,amsmath}
\newtheorem{thm}{Theorem}[section]

 \journalname{my journal}
\begin{document}
\title{Global  existence results and duality for  non-linear models of plates and shells}

\titlerunning{On existence and duality for a plate model}        

\author{Fabio Silva Botelho }


\institute{Fabio Silva Botelho \at Department of Mathematics \\
              Federal University of Santa Catarina, SC - Brazil \\
              Tel.: +55-48-3721-3663\\
              \email{fabio.botelho@ufsc.br}  }        


\maketitle

\begin{abstract}
In this article firstly we develop a new proof for global existence of minimizers for the Kirchhoff-Love plate model.

We also present a duality principle and relating sufficient optimality conditions for  such a  variational plate model.

In a second step, we present a global existence result for a non-linear model of shells. For this model, we also develop a duality principle
and concerning sufficient conditions of optimality. \\
\end{abstract}

\section{Introduction} In the present work, in a first step,  we develop  a new existence proof and  a dual variational formulation for the Kirchhoff-Love thin plate model. Previous results on existence in mathematical elasticity and related models may be found in \cite{903,[3],[4]}.

At this point we refer to the exceptionally important article "A contribution to contact problems for a class of solids and structures" by
W.R. Bielski and J.J. Telega, \cite{85},  published in 1985,  as the first one to successfully  apply and generalize the convex analysis approach to a model in non-convex and non-linear mechanics.

The present work is, in some sense, a kind of extension  of this previous work \cite{85} and others such as \cite{2900}, which greatly influenced and
inspired my work and recent book \cite{120}.

Here we highlight that such earlier results establish the complementary energy under the hypothesis of positive definiteness of the membrane force tensor at a critical point (please see  \cite{85,2900,77} for details).

We have obtained a dual variational formulation which allows the global optimal point in question not to be positive definite (for  related results see F.Botelho \cite{120}), but also not necessarily negative definite. The approach developed also includes sufficient conditions of optimality for the primal problem.
It is worth mentioning that the standard tools of convex analysis used in this text may be found in \cite{[6],120,29}, for example.

At this point we start to describe the primal formulation.

    Let $\Omega\subset\mathbb{R}^2$ be an open, bounded, connected set  which  represents the middle surface of a plate
of thickness $h$. The boundary of $\Omega$, which is assumed to be regular (Lipschitzian), is
denoted by $\partial \Omega$. The vectorial basis related to the cartesian
system $\{x_1,x_2,x_3\}$ is denoted by $( \textbf{a}_\alpha,
\textbf{a}_3)$, where $\alpha =1,2$ (in general Greek indices stand
for 1 or 2), and where $\textbf{a}_3$ is the vector normal to $\Omega$, whereas $\textbf{a}_1$ and $\textbf{a}_2$ are orthogonal vectors parallel to $\Omega.$  Also,
$\textbf{n}$ is
the outward normal to the plate surface.

    The displacements will be denoted by
$$ \hat{\textbf{u}}=\{\hat{u}_\alpha,\hat{u}_3\}=\hat{u}_\alpha
\textbf{a}_\alpha+ \hat{u}_3 \textbf{a}_3.
$$

The Kirchhoff-Love relations are
\begin{eqnarray}&&\hat{u}_\alpha(x_1,x_2,x_3)=u_\alpha(x_1,x_2)-x_3
w(x_1,x_2)_{,\alpha}\; \nonumber \\ && \text{ and }
\hat{u}_3(x_1,x_2,x_3)=w(x_1,x_2).
\end{eqnarray}

Here $ -h/2\leq x_3 \leq h/2$ so that we have
$u=(u_\alpha,w)\in U$ where
\begin{eqnarray}
U&=&\left\{(u_\alpha,w)\in W^{1,2}(\Omega; \mathbb{R}^2) \times
W^{2,2}(\Omega) , \right.\nonumber \\ &&\; u_\alpha=w=\frac{\partial w}{\partial \textbf{n}}=0 \left.
\text{ on } \partial \Omega \right\} \nonumber \\ &=&W_0^{1,2}(\Omega; \mathbb{R}^2) \times W_0^{2,2}(\Omega).\nonumber\end{eqnarray}
It is worth emphasizing that the boundary conditions here specified refer to a clamped plate.

We define the operator $\Lambda: U \rightarrow Y \times Y$, where $Y=Y^*=L^2(\Omega; \mathbb{R}^{2 \times 2})$, by
$$\Lambda(u)=\{\gamma(u), \kappa(u)\},$$
$$\gamma_{\alpha\beta}(u)= \frac{u_{\alpha,\beta}+u_{\beta,\alpha}}{2}+\frac{w_{,\alpha} w_{,\beta}}{2},$$
$$\kappa_{\alpha \beta}(u)=-w_{,\alpha \beta}.$$
The constitutive relations are given by
\begin{equation}
N_{\alpha\beta}(u)=H_{\alpha\beta\lambda\mu} \gamma_{\lambda\mu}(u),
\end{equation}
\begin{equation}
M_{\alpha\beta}(u)=h_{\alpha\beta\lambda\mu} \kappa_{\lambda\mu}(u),
\end{equation}
where: $\{H_{\alpha\beta\lambda\mu}\}$
 and
 $\{h_{\alpha\beta\lambda\mu}=\frac{h^2}{12}H_{\alpha\beta\lambda\mu}\}$,
 are symmetric positive definite   fourth order tensors. From now on, we denote $\{\overline{H}_{\alpha\beta \lambda \mu}\}=\{H_{\alpha\beta \lambda \mu}\}^{-1}$ and $\{\overline{h}_{\alpha\beta \lambda \mu}\}=\{h_{\alpha\beta \lambda \mu}\}^{-1}$.

 Furthermore
 $\{N_{\alpha\beta}\}$ denote the membrane force tensor and
 $\{M_{\alpha\beta}\}$ the moment one.
    The plate stored energy, represented by $(G\circ
    \Lambda):U\rightarrow\mathbb{R}$ is expressed by
 \begin{equation}\label{80} (G\circ\Lambda)(u)=\frac{1}{2}\int_{\Omega}
  N_{\alpha\beta}(u)\gamma_{\alpha\beta}(u)\;dx+\frac{1}{2}\int_{\Omega}
  M_{\alpha\beta}(u)\kappa_{\alpha\beta}(u)\;dx
  \end{equation}
 and the external work, represented by $F:U\rightarrow\mathbb{R}$, is given by
   \begin{equation}\label{81} F(u)=\langle w,P \rangle_{L^2(\Omega)}+\langle u_\alpha,P_\alpha \rangle_{L^2(\Omega)}
,\end{equation} where $P, P_1, P_2 \in L^2(\Omega)$ are external loads in the directions $\textbf{a}_3$, $\textbf{a}_1$ and $\textbf{a}_2$ respectively. The potential energy, denoted by
$J:U\rightarrow\mathbb{R}$ is expressed by:
$$
J(u)=(G\circ\Lambda)(u)-F(u)
$$

Finally, we also emphasize from now on, as their meaning are clear, we may denote $L^2(\Omega)$ and $L^2(\Omega; \mathbb{R}^{2 \times 2})$ simply by $L^2$, and the respective norms by $\|\cdot \|_2.$ Moreover derivatives are always understood in the distributional sense, $\mathbf{0}$ may denote the zero vector in  appropriate Banach spaces and, the following and relating notations are used:
$$w_{,\alpha\beta}=\frac{\partial^2 w}{\partial x_\alpha \partial x_\beta},$$
$$u_{\alpha,\beta}=\frac{\partial u_\alpha}{\partial x_\beta},$$
$$N_{\alpha\beta,1}=\frac{\partial N_{\alpha\beta}}{\partial x_1},$$
and
$$N_{\alpha\beta,2}=\frac{\partial N_{\alpha\beta}}{\partial x_2}.$$
\section{On the existence of a global minimizer}

At this point we present an existence result concerning the Kirchhoff-Love plate model.

We start with the following two remarks.
\begin{remark}\label{us123W}Let $\{P_\alpha\} \in L^\infty(\Omega;\mathbb{R}^2)$. We may easily obtain by appropriate Lebesgue integration $\{\tilde{T}_{\alpha \beta}\}$ symmetric and such that

 $$\tilde{T}_{\alpha\beta,\beta}=-P_\alpha, \text{ in } \Omega.$$

 Indeed, extending $\{P_\alpha\}$ to zero outside $\Omega$ if necessary, we may set
 $$\tilde{T}_{11}(x,y)=-\int_0^xP_1(\xi,y)\;d\xi,$$
  $$\tilde{T}_{22}(x,y)=-\int_0^yP_2(x,\xi)\;d\xi,$$
  and $$\tilde{T}_{12}(x,y)=\tilde{T}_{21}(x,y)=0, \text{ in } \Omega.$$

 Thus, we may choose a $C>0$ sufficiently big, such that $$\{T_{\alpha\beta}\}=\{\tilde{T}_{\alpha\beta} +C \delta_{\alpha\beta}\}$$ is positive definite $\text{ in } \Omega$, so that

 $$T_{\alpha\beta,\beta}=\tilde{T}_{\alpha\beta,\beta}=-P_\alpha,$$

 where $$\{\delta_{\alpha\beta}\}$$ is the Kronecker delta.

 So, for the kind of boundary conditions of the next theorem, we do NOT have any restriction for the $\{P_\alpha\}$ norm.

 Summarizing, the next result is new and it is really a step forward concerning the previous one in Ciarlet \cite{[3]}. We emphasize this result and
 its proof through such a tensor $\{T_{\alpha\beta}\}$ are new, even though the final part of the proof is established through a standard procedure in
 the calculus of variations.

 About the other existence result for plates, its proof through the tensor well specified $\{(T_0)_{\alpha\beta}\}$ is also new, even though the final  part
 of such a proof  is also performed through a standard procedure.

 A similar remark is valid for the existence result for the model of shells, which is also established through a tensor $T_0$ properly specified.

 Finally, the duality principles and concerning optimality conditions are established through new functionals. Similar results may be found
 in \cite{120}.
\end{remark}

\begin{remark} Specifically about the existence of the tensor $T_0$ relating Theorem \ref{us123a}, we recall the following well known duality principle
of the calculus of variations
\begin{eqnarray} &&\inf_{T=\{T_{\alpha\beta}\} \in B^*} \left\{\frac{1}{2}\|T\|_2^2 \right\}\nonumber \\ &=&
\sup_{ \{u_\alpha\} \in \tilde{U}} \left\{ -\frac{1}{2}\int_\Omega \nabla u_\alpha \cdot \nabla u_\alpha \;dx +\langle u_\alpha, P_\alpha\rangle_{L^2(\Omega)}
+\langle u_\alpha, P^t_\alpha\rangle_{L^2(\Gamma_t)}\right\}.
\end{eqnarray}
Here $$B^*=\{T \in L^2(\Omega;\mathbb{R}^2)\;:\; T_{\alpha\beta, \beta}+P_\alpha=0, \; \text{ in } \Omega,\; T_{\alpha\beta}n_\beta-P^t_\alpha=0, \text{ on } \Gamma_t\},$$
and
$$\tilde{U}=\{\{u_\alpha\} \in W^{1,2}(\Omega;\mathbb{R}^2)\;:\; u_\alpha=0 \text{ on } \Gamma_0\}.$$

We also recall that the existence of a unique solution for both these primal and dual convex formulations is a well known result of the duality theory in the calculus of variations. Please, see related results in \cite{[6]}.

A similar duality principle may be established for the case of Theorem \ref{us123b}.
\end{remark}
\begin{thm}\label{usp10}Let $\Omega \subset \mathbb{R}^2$ be an open, bounded, connected set with a Lipschitzian boundary denoted by $\partial \Omega=\Gamma.$ Suppose $(G \circ \Lambda):U \rightarrow \mathbb{R}$ is defined by
$$G(\Lambda u)=G_1(\gamma (u))+G_2(\kappa (u)), \; \forall u \in U,$$
where
$$G_1(\gamma u)= \frac{1}{2}\int_\Omega H_{\alpha\beta\lambda\mu} \gamma_{\alpha\beta}(u)\gamma_{\lambda \mu}(u)\;dx,$$
and
$$G_2(\kappa u)=\frac{1}{2}\int_\Omega h_{\alpha\beta\lambda\mu}\kappa_{\alpha\beta}(u)\kappa_{\lambda \mu}(u)\;dx,$$
where
$$\Lambda(u)=(\gamma(u),\kappa(u))=(\{\gamma_{\alpha\beta}(u)\},\{\kappa_{\alpha\beta}(u)\}),$$
$$\gamma_{\alpha\beta}(u)=\frac{u_{\alpha,\beta}+u_{\beta,\alpha}}{2}+\frac{w_{,\alpha} w_{,\beta}}{2},$$
$$\kappa_{\alpha\beta}(u)=-w_{,\alpha\beta},$$
and where
$$U=\{u=(u_1,u_2,w)\in W^{1,2}(\Omega;\mathbb{R}^2) \times W^{2,2}(\Omega) \;:\; u_1=u_2=w=0 \text{ on } \partial \Omega\}.$$

We also define,
\begin{eqnarray} F_1(u)&=& \langle w,P\rangle_{L^2}+\langle u_\alpha, P_\alpha \rangle_{L^2} \nonumber \\
&\equiv& \langle u,\mathbf{f}\rangle_{L^2}, \end{eqnarray}
where
$$\mathbf{f}=(P_\alpha,P) \in L^\infty(\Omega;\mathbb{R}^3).$$

Let $J:U \rightarrow \mathbb{R}$ be defined by
$$J(u)=G(\Lambda u)-F_1(u),\; \forall u \in U.$$
Assume there exists $\{c_{\alpha\beta}\} \in \mathbb{R}^{2 \times 2}$ such that $c_{\alpha\beta}>0,\; \forall \alpha,\beta \in \{1,2\}$ and $$G_2(\kappa(u)) \geq c_{\alpha\beta}\|w_{,\alpha\beta}\|_2^2,\; \forall u \in U.$$

Under such hypotheses, there exists $u_0 \in U$ such that
$$J(u_0)=\min_{u \in U} J(u).$$
\end{thm}
\begin{proof} Observe that we may find $\mathbf{T}_\alpha=\{(T_\alpha)_\beta\}$ such that
$$div \mathbf{T}_{\alpha}=T_{\alpha\beta,\beta}=-P_\alpha$$ an also such that $\{T_{\alpha\beta}\}$ is positive definite and symmetric (please, see Remark \ref{us123W}).

Thus defining \begin{equation}\label{a.1}v_{\alpha\beta}(u)= \frac{u_{\alpha,\beta}+u_{\beta,\alpha}}{2}+\frac{1}{2}w_{,\alpha}w_{,\beta},\end{equation} we obtain
\begin{eqnarray}\label{usp1}
J(u)&=&G_1(\{v_{\alpha\beta}(u)\})+G_2(\kappa(u))-\langle P_\alpha,u_\alpha \rangle_{L^2}-\langle w, P\rangle_{L^2} \nonumber \\ &=& G_1(\{v_{\alpha\beta}(u)\})+G_2(\kappa(u))+\langle T_{\alpha\beta,\beta},u_\alpha \rangle_{L^2}-\langle w, P\rangle_{L^2}
\nonumber \\ &=& G_1(\{v_{\alpha\beta}(u)\})+G_2(\kappa(u))-\left\langle T_{\alpha\beta},\frac{u_{\alpha,\beta}+u_{\beta,\alpha}}{2} \right\rangle_{L^2}-\langle w, P\rangle_{L^2} \nonumber \\ &=& G_1(\{v_{\alpha\beta}(u)\})+G_2(\kappa(u))-\left\langle T_{\alpha\beta},v_{\alpha\beta}(u)-\frac{1}{2}w_{,\alpha}w_{,\beta} \right\rangle_{L^2}-\langle w, P\rangle_{L^2} \nonumber \\ &\geq& c_{\alpha\beta}\|w_{,\alpha\beta}\|_2^2+\frac{1}{2}\left\langle T_{\alpha\beta},w_{,\alpha}w_{,\beta} \right\rangle_{L^2}-\langle w, P\rangle_{L^2}+G_1(\{v_{\alpha\beta}(u)\})\nonumber \\ &&-\langle T_{\alpha\beta},v_{\alpha\beta}(u)\rangle_{L^2}.
\end{eqnarray}

From this, since $\{T_{\alpha\beta}\}$ is positive definite,  clearly $J$ is bounded below.

Let $\{u_n\} \in U$ be a minimizing sequence for $J$. Thus there exists $\alpha_1 \in \mathbb{R}$ such that
$$\lim_{n \rightarrow \infty}J(u_n)= \inf_{u \in U} J(u)=\alpha_1.$$

From (\ref{usp1}), there exists $K_1>0$ such that
$$\|(w_n)_{,\alpha\beta}\|_2< K_1,\forall \alpha,\beta \in \{1,2\},\;n \in \mathbb{N}.$$

Therefore, there exists $w_0 \in W^{2,2}(\Omega)$ such that, up to a subsequence not relabeled,
$$(w_n)_{,\alpha\beta} \rightharpoonup (w_0)_{,\alpha\beta},\; \text{ weakly in } L^2,$$
 $\forall \alpha,\beta \in \{1,2\}, \text{ as } n \rightarrow \infty.$

Moreover, also up to a subsequence not relabeled,
\begin{equation}\label{a.2}(w_n)_{,\alpha} \rightarrow (w_0)_{,\alpha},\; \text{ strongly in } L^2  \text{ and } L^4,\end{equation}
 $\forall \alpha, \in \{1,2\}, \text{ as } n \rightarrow \infty.$

Also from (\ref{usp1}), there exists $K_2>0$ such that,
$$\|(v_n)_{\alpha\beta}(u)\|_2< K_2,\forall \alpha,\beta \in \{1,2\},\;n \in \mathbb{N},$$
and thus, from this, (\ref{a.1}) and (\ref{a.2}), we may infer that
there exists $K_3>0$ such that
$$\|(u_n)_{\alpha,\beta}+(u_n)_{\beta,\alpha}\|_2< K_3,\forall \alpha,\beta \in \{1,2\},\;n \in \mathbb{N}.$$

From this and Korn's inequality, there exists $K_4>0$ such that

$$\|u_n\|_{W^{1,2}(\Omega;\mathbb{R}^2)} \leq K_4, \; \forall n \in \mathbb{N}.$$
So, up to a subsequence not relabeled, there exists $\{(u_0)_\alpha\} \in W^{1,2}(\Omega, \mathbb{R}^2),$ such that
$$(u_n)_{\alpha,\beta}+(u_n)_{\beta,\alpha} \rightharpoonup (u_0)_{\alpha,\beta}+(u_0)_{\beta,\alpha},\; \text{ weakly in } L^2,$$
 $\forall \alpha,\beta \in \{1,2\}, \text{ as } n \rightarrow \infty,$
and,
$$(u_n)_{\alpha} \rightarrow (u_0)_{\alpha},\; \text{ strongly in } L^2,$$
 $\forall \alpha \in \{1,2\}, \text{ as } n \rightarrow \infty.$

Moreover, the boundary conditions satisfied by the subsequences are also satisfied for $w_0$ and $u_0$ in a trace sense, so that
 $$u_0=((u_0)_\alpha,w_0) \in U.$$

From this, up to a subsequence not relabeled, we get
$$\gamma_{\alpha\beta}(u_n) \rightharpoonup \gamma_{\alpha\beta}(u_0), \text{ weakly in } L^2,$$
$\forall \alpha,\beta \in \{1,2\},$
and
$$\kappa_{\alpha\beta}(u_n) \rightharpoonup \kappa_{\alpha\beta}(u_0), \text{ weakly in } L^2,$$
$\forall \alpha,\beta \in \{1,2\}.$

Therefore, from the convexity of $G_1$ in $\gamma$ and $G_2$ in $\kappa$ we obtain
\begin{eqnarray}
\inf_{u \in U}J(u)&=& \alpha_1 \nonumber \\ &=& \liminf_{n \rightarrow \infty}J(u_n) \nonumber \\ &\geq&
J(u_0).
\end{eqnarray}

Thus, $$J(u_0)=\min_{u\in U}J(u).$$

The proof is complete.
\end{proof}

\section{Existence of a minimizer for the plate model for a more general case}

At this point we present an existence result for a more general case.
\begin{thm}\label{us123a}Consider the statements and assumptions concerning the plate model described in the last section.

More specifically, consider the functional $J:U \rightarrow \mathbb{R}$ given, as above described, by

\begin{eqnarray}J(u)&=& W(\gamma(u),\kappa(u))-\langle P_\alpha u_\alpha\rangle_{L^2}
\nonumber \\ &&-\langle w,P\rangle_{L^2}-\langle P^t_\alpha, u_\alpha \rangle_{L^2(\Gamma_t)} \nonumber \\&&
-\langle P^t,w \rangle_{L^2(\Gamma_t)},\end{eqnarray}
where,
 \begin{eqnarray} U&=&\{u=(u_\alpha,w)=(u_1,u_2,w) \in W^{1,2}(\Omega; \mathbb{R}^2) \times W^{2,2}(\Omega)\;:
\nonumber \\ && u_\alpha=w= \frac{\partial w}{\partial \mathbf{n}}=0, \text{ on } \Gamma_0\},
\end{eqnarray}
where $\partial \Omega =\Gamma_0 \cup \Gamma_t$ and the Lebesgue measures $$m_{\Gamma}(\Gamma_0 \cap \Gamma_t)=0,$$
and $$m_\Gamma (\Gamma_0)>0.$$

Let $T_0$ be such that,
$$\|T_0\|_2^2=\min_{ T \in L^2(\Omega;\mathbb{R}^{2 \times 2})}\{\|T\|_2^2\},$$
subject to $$T_{\alpha\beta,\beta}+P_\alpha=0 \text{ in }\Omega,$$
$$(T_0)_{\alpha\beta} \mathbf{n}_\beta -P^t_\alpha=0, \text{ on } \Gamma_t.$$
Assume $\|P_\alpha\|_2$ and $\|P^t_\alpha\|$ are small enough so that \begin{equation}\label{e.1}J_1(u) \rightarrow +\infty, \text{ as } \langle w_{,\alpha\beta},w_{,\alpha\beta} \rangle_{L^2}  \rightarrow +\infty,\end{equation}
where
\begin{eqnarray}J_1(u)&=&G_2(\kappa(u))+\frac{1}{2}\langle (T_0)_{\alpha\beta}, w_{,\alpha}w_{,\beta} \rangle_{L^2}
 \nonumber \\ &&-\langle P,w \rangle_{L^2}-\langle P^t,w \rangle_{L^2(\Gamma_t)}.\end{eqnarray}

Under such hypotheses, there exists $u_0 \in U$ such that,
$$J(u_0)=\min_{u \in U} \{J(u)\}.$$
\end{thm}

\begin{proof}

Observe that defining \begin{equation}\label{a.3}v_{\alpha\beta}(u)= \frac{u_{\alpha,\beta}+u_{\beta,\alpha}}{2}+\frac{1}{2}w_{,\alpha}w_{,\beta},\end{equation}
we have,

\begin{eqnarray}\label{usp90}
J(u)&=&  G_1(v(u))+G_2(\kappa (u))\nonumber \\ &&-\langle P_\alpha,u_\alpha \rangle_{L^2}-\langle P,w \rangle_{L^2}
\nonumber \\ && -\langle P^t_\alpha, u_\alpha \rangle_{L^2(\Gamma_t)}-\langle P^t,w \rangle_{L^2(\Gamma_t)}
\nonumber \\ &=& G_1(v(u))+G_2(\kappa(u))+\langle (T_0)_{\alpha\beta,\beta},u_\alpha \rangle_{L^2}
\nonumber \\ &&-\langle P,w \rangle_{L^2}
\nonumber \\ && -\langle P^t_\alpha, u_\alpha \rangle_{L^2(\Gamma_t)}-\langle P^t,w \rangle_{L^2(\Gamma_t)}
 \nonumber \\ &=& G_1(v(u))+G_2(\kappa(u))-\left\langle (T_0)_{\alpha\beta},\frac{u_{\alpha,\beta}+u_{\beta,\alpha}}{2} \right\rangle_{L^2}
\nonumber \\ && +\langle (T_0)_{\alpha\beta}\mathbf{n}_\beta, u_\alpha \rangle_{L^2(\Gamma_t)}-\langle P,w \rangle_{L^2}
\nonumber \\ && -\langle P^t_\alpha, u_\alpha \rangle_{L^2(\Gamma_t)}-\langle P^t,w \rangle_{L^2(\Gamma_t)}
\nonumber \\ &=& G_1(v(u))+G_2(\kappa(u))-\left\langle (T_0)_{\alpha\beta},\frac{u_{\alpha,\beta}+u_{\beta,\alpha}}{2} \right\rangle_{L^2}
\nonumber \\ &&-\langle P,w \rangle_{L^2}
-\langle P^t,w \rangle_{L^2(\Gamma_t)} \nonumber \\ &=&G_1(v(u))+G_2(\kappa(u))-\left\langle (T_0)_{\alpha\beta},
v_{\alpha\beta}(u)-\frac{1}{2} w_{,\alpha}w_{,\beta} \right\rangle_{L^2}
\nonumber \\ &&-\langle P,w \rangle_{L^2}
-\langle P^t,w \rangle_{L^2(\Gamma_t)} \nonumber \\ &=& J_1(u)+G_1(v(u))-\langle (T_0)_{\alpha\beta}, v_{\alpha\beta}(u) \rangle_{L^2}
\end{eqnarray}

From this and the  hypothesis (\ref{e.1}), $J$ is bounded below. So, there exists $\alpha_1 \in \mathbb{R}$ such that
$$\alpha_1=\inf_{u \in U} J(u).$$

Let $\{u_n\}$ be a minimizing sequence for $J$.

From (\ref{usp90}) and also from the hypothesis (\ref{e.1}), $\{\|w_n\|_{2,2}\}$  is bounded.

So there exists $w_0 \in U$  such that, up to a not relabeled subsequence,
$$(w_n)_{,\alpha\beta} \rightharpoonup (w_0)_{,\alpha\beta}
\text{ weakly in } L^2(\Omega),\; \forall \alpha,\beta \in \{1,2\},$$
\begin{equation}\label{a.4}(w_n)_{,\alpha} \rightarrow (w_0)_{,\alpha}
\text{ strongly in } L^2(\Omega),\; \forall \alpha\in \{1,2\},\end{equation}

Also from (\ref{usp1}), there exists $K_2>0$ such that,
$$\|(v_n)_{\alpha\beta}(u)\|_2< K_2,\forall \alpha,\beta \in \{1,2\},\;n \in \mathbb{N},$$
and thus, from this, (\ref{a.3}) and (\ref{a.4}), we may infer that
there exists $K_3>0$ such that
$$\|(u_n)_{\alpha,\beta}(u)+(u_n)_{\beta,\alpha}(u)\|_2< K_3,\forall \alpha,\beta \in \{1,2\},\;n \in \mathbb{N}.$$
From this and Korn's inequality, there exists $K_4>0$ such that

$$\|u_n\|_{W^{1,2}(\Omega;\mathbb{R}^2)} \leq K_4, \; \forall n \in \mathbb{N}.$$
So, up to a subsequence not relabeled, there exists $\{(u_0)_\alpha\} \in W^{1,2}(\Omega, \mathbb{R}^2),$ such that
$$(u_n)_{\alpha,\beta}+(u_n)_{\beta,\alpha} \rightharpoonup (u_0)_{\alpha,\beta}+(u_0)_{\beta,\alpha},\; \text{ weakly in } L^2,$$
 $\forall \alpha,\beta \in \{1,2\}, \text{ as } n \rightarrow \infty,$
and,
$$(u_n)_{\alpha} \rightarrow (u_0)_{\alpha},\; \text{ strongly in } L^2,$$
 $\forall \alpha \in \{1,2\}, \text{ as } n \rightarrow \infty.$

Moreover, the boundary conditions satisfied by the subsequences are also satisfied for $w_0$ and $u_0$ in a trace sense, so that
 $$u_0=((u_0)_\alpha,w_0) \in U.$$

From this, up to a subsequence not relabeled, we get
$$\gamma_{\alpha\beta}(u_n) \rightharpoonup \gamma_{\alpha\beta}(u_0), \text{ weakly in } L^2,$$
$\forall \alpha,\beta \in \{1,2\},$
and
$$\kappa_{\alpha\beta}(u_n) \rightharpoonup \kappa_{\alpha\beta}(u_0), \text{ weakly in } L^2,$$
$\forall \alpha,\beta \in \{1,2\}.$

Therefore, from the convexity of $G_1$ in $\gamma$ and $G_2$ in $\kappa$ we obtain
\begin{eqnarray}
\inf_{u \in U}J(u)&=& \alpha_1 \nonumber \\ &=& \liminf_{n \rightarrow \infty}J(u_n) \nonumber \\ &\geq&
J(u_0).
\end{eqnarray}

Thus, $$J(u_0)=\min_{u\in U}J(u).$$

The proof is complete.

\end{proof}

\section{The main duality principle}

 In this section we present a duality principle for the plate model in question.

 For such a result, we emphasize the dual variational formulation
 is concave.
\begin{remark} In the proofs relating our duality principles we apply a very well known result found in Toland \cite{12}.

Indeed, for $$\{N_{\alpha\beta}\} \in L^2(\Omega; \mathbb{R}^{2 \times 2}),$$ assume
$$F_5(w)-G_5(\{w_\alpha\})>0,\; \forall w \in W_0^{2,2}(\Omega) \; \text{ such that } w \neq \mathbf{0},$$
where here $$F_5(w)=\frac{1}{2}\int_\Omega h_{\alpha\beta\lambda \mu} w_{,\alpha\beta} w_{,\lambda \mu}\;dx + \frac{K}{2} \langle w_{,\alpha}, w_{,\alpha} \rangle_{L^2},$$
and
$$G_5(\{w_\alpha\})=-\frac{1}{2}\int_\Omega N_{\alpha\beta} w_{,\alpha}w_{,\beta}\;dx+ \frac{K}{2} \langle w_{,\alpha}, w_{,\alpha} \rangle_{L^2},$$
where $K>0$ is supposed to be sufficiently big so that $G_5$ is convex in $w$.

Thus,
\begin{eqnarray}F_5(w)-G_5(\{w_\alpha\})&=&\frac{1}{2}\int_\Omega h_{\alpha\beta\lambda \mu} w_{,\alpha\beta} w_{,\lambda \mu}\;dx
\nonumber \\ &&+\frac{1}{2}\int_\Omega N_{\alpha\beta} w_{,\alpha}w_{,\beta}\;dx>0, \end{eqnarray}
$\forall w \in W_0^{2,2}(\Omega)
\; \text{ such that } w \neq \mathbf{0}.$

Therefore,

\begin{eqnarray}&&-\langle w_\alpha, z^*_\alpha \rangle_{L^2}+F_5(w)+\sup_{v_2 \in L^2}\{\langle (v_2)_\alpha, z^*_\alpha \rangle_{L^2}-G_5(\{(v_2)_\alpha\})
\nonumber \\ &=& -\langle w_\alpha, z^*_\alpha \rangle_{L^2}+F_5(w)+\frac{1}{2} \int_\Omega \overline{(-N_{\alpha\beta})^K} z^*_\alpha z^*_\beta\;dx >0,\end{eqnarray}
$$\forall w \in W_0^{2,2}(\Omega)
\; \text{ such that } w \neq \mathbf{0},$$
so that

\begin{eqnarray}&&\inf_{ w \in W_0^{2,2}(\Omega)}\{-\langle w_\alpha, z^*_\alpha \rangle_{L^2}+F_5(w)\}+\frac{1}{2} \int_\Omega \overline{(-N_{\alpha\beta})^K} z^*_\alpha z^*_\beta\;dx \nonumber \\ &=&  -F^*_5(z^*)+\frac{1}{2} \int_\Omega \overline{(-N_{\alpha\beta})^K} z^*_\alpha z^*_\beta\;dx  \geq 0,\end{eqnarray}
$\forall z^* \in L^2.$

Indeed, from the general result in Toland \cite{12},
we have

\begin{eqnarray}&&\inf_{z^* \in L^2}\left\{-F^*_5(z^*)+\frac{1}{2} \int_\Omega \overline{(-N_{\alpha\beta})^K} z^*_\alpha z^*_\beta\;dx \right\}
\nonumber \\ &=& \inf_{ w \in W_0^{2,2}} \{F_5(w)-G_5(\{w_{,\alpha}\}) \nonumber \\ &\leq&  F_5(w)-G_5(\{w_{,\alpha}\})
\nonumber \\ &=& \frac{1}{2}\int_\Omega h_{\alpha\beta\lambda \mu} w_{,\alpha\beta} w_{,\lambda \mu}\;dx
\nonumber \\ &&+\frac{1}{2}\int_\Omega N_{\alpha\beta} w_{,\alpha}w_{,\beta}\;dx,\;\; \forall w \in W_0^{2,2}(\Omega). \end{eqnarray}
\end{remark}

At this point we enunciate and prove our main duality principle.

\begin{thm} Let $\Omega \subset \mathbb{R}^2$ be an open, bounded, connected set with a Lipschitzian boundary denoted by $\partial \Omega=\Gamma.$ Suppose $(G \circ \Lambda):U \rightarrow \mathbb{R}$ is defined by
$$G(\Lambda u)=G_1(\gamma (u))+G_2(\kappa (u)), \; \forall u \in U,$$
where
$$G_1(\gamma (u))= \frac{1}{2}\int_\Omega H_{\alpha\beta\lambda\mu} \gamma_{\alpha\beta}(u)\gamma_{\lambda \mu}(u)\;dx,$$
and
$$G_2(\kappa (u))=\frac{1}{2}\int_\Omega h_{\alpha\beta\lambda\mu}\kappa_{\alpha\beta}(u)\kappa_{\lambda \mu}(u)\;dx,$$
where
$$\Lambda(u)=(\gamma(u),\kappa(u))=(\{\gamma_{\alpha\beta}(u)\},\{\kappa_{\alpha\beta}(u)\}),$$
$$\gamma_{\alpha\beta}(u)=\frac{u_{\alpha,\beta}+u_{\beta,\alpha}}{2}+\frac{w_{,\alpha} w_{,\beta}}{2},$$
$$\kappa_{\alpha\beta}(u)=-w_{\alpha\beta}.$$
Here,
$$u=(u_1,u_2,w)=(u_\alpha,w) \in U=W_0^{1,2}(\Omega;\mathbb{R}^2) \times W_0^{2,2}(\Omega).$$

We also define,
\begin{eqnarray} F_1(u)&=& \langle w,P\rangle_{L^2}+\langle u_\alpha, P_\alpha \rangle_{L^2} \nonumber \\
&\equiv& \langle u,\mathbf{f}\rangle_{L^2}, \end{eqnarray}
where
$$\mathbf{f}=(P_\alpha,P) \in L^2(\Omega;\mathbb{R}^3).$$

Let $J:U \rightarrow \mathbb{R}$ be defined by
$$J(u)=G(\Lambda u)-F_1(u),\; \forall u \in U.$$

Under such hypotheses,
\begin{eqnarray}&& \inf_{u \in U} J(u) \nonumber \\ &\geq&
\sup_{v^* \in A^*}\{\inf_{z^* \in Y_2^*}\{-F^*(z^*,Q)+G^*(z^*,N)\}\},
\end{eqnarray}
where,
$$F(u)=G_2(\kappa(u))+\frac{K}{2} \langle w_{,\alpha},w_{,\alpha} \rangle_{L^2}, \forall u \in U.$$

Moreover, $F^*:[Y_2^*]^2 \rightarrow \mathbb{R}$ is defined by,
$$F^*(z^*,Q)=\sup_{u \in U}\{ \langle z^*_\alpha+Q_\alpha ,w_{,\alpha} \rangle_{L^2}-F(u)\},\; \forall z^* \in [Y_2^*]^2.$$

Also,

\begin{eqnarray} G(v)&=&-\frac{1}{2}\int_\Omega H_{\alpha\beta\lambda\mu}\left[(v_1)_{\alpha\beta}+\frac{(v_2)_\alpha (v_2)_\beta}{2}\right]
\left[(v_1)_{\lambda\mu}+\frac{(v_2)_\lambda (v_2)_\mu}{2}\right]\;dx \nonumber \\ &&+\frac{K}{2}\langle (v_2)_\alpha,(v_2)_\alpha \rangle_{L^2},
\end{eqnarray}
\begin{eqnarray}
&&G^*(z^*,N) \nonumber \\
&=& \sup_{v_2 \in Y_2}\{\inf_{v_1 \in Y_1}\{\;\langle N_{\alpha\beta}, (v_1)_{\alpha\beta} \rangle_{L^2}+\langle Q_\alpha+z^*_\alpha, (v_2)_\alpha \rangle_{L^2}-G(v)\}\}
\nonumber \\ &=& \frac{1}{2}\int_\Omega \overline{(-N_{\alpha\beta})^K}z^*_\alpha z^*_\beta\;dx \nonumber \\ &&-
\frac{1}{2}\int_\Omega \overline{H}_{\alpha\beta\lambda\mu}N_{\alpha\beta}N_{\lambda \mu}\;dx,\end{eqnarray}
if $v^*=(Q,N) \in A_3,$ where
$$A_3=\{v^* \in Y^*\;:\;\{(-N_{\alpha\beta})^K\}\text{ is positive definite in } \Omega\},$$

 \begin{equation} \{(-N_{\alpha\beta})^K\}= \left\{
\begin{array}{lr}
 -N_{11}+K  &  -N_{12}
 \\
 -N_{21} & - N_{22}+K
 \end{array} \right\},\end{equation}
and
 $$\overline{(-N_{\alpha\beta})^K}=\{(-N_{\alpha\beta})^K\}^{-1}.$$

Moreover,

$$A^*=A_1 \cap A_2\cap A_3 \cap A_4,$$
where $$A_1=\{v^*=(N,Q) \in Y^* \;:\; N_{\alpha\beta,\beta}+P_\alpha=0, \text{ in } \Omega,\; \forall \alpha \in \{1,2\}\}$$
and
$$A_2=\{v^* \in Y^*\;:\; Q_{\alpha,\alpha}+P=0, \text{ in } \Omega\}.$$
Also,
\begin{eqnarray}A_4&=&\{v^*=(Q,N) \in Y^*\;:\; \hat{J}^*(z^*)>0, \nonumber \\ &&\; \forall z^* \in Y_2^*, \; \text{ such that } z^* \neq \mathbf{0}\},
\end{eqnarray}
where, \begin{eqnarray}\hat{J}^*(z^*)&=& -F^*(z^*,\mathbf{0})+G^*(z^*,\mathbf{0}) \nonumber \\ &=&
-F^*(z^*,\mathbf{0})+\frac{1}{2}\int_S \overline{(-N_{\alpha\beta})^K}\; z^*_\alpha z^*_\beta\;dx,\; \forall z^* \in Y_2^*.\nonumber\end{eqnarray}

Here, $$Y^*=Y=L^2(\Omega; \mathbb{R}^2) \times L^2(\Omega;\mathbb{R}^{2 \times 2}),$$
$$Y_1^*=Y_1=L^2(\Omega;\mathbb{R}^{2\times 2}),$$ and $$Y_2^*=Y_2=L^2(\Omega;\mathbb{R}^2),$$

 Finally, denoting

 $$J^*(v^*,z^*)=-F^*(z^*,Q)+G^*(z^*,N), \forall (v^*,z^*) \in A^* \times Y_2^*,$$
 and
 $$\tilde{J}^*(v^*)=\inf_{z^* \in Y_2^*} J^*(v^*,z^*), \forall v^* \in A^*,$$
 suppose there exist $v_0^*=(N_0,Q_0) \in A^*$, $z_0^* \in Y_2^*$ and $u_0 \in U$ such that
 $$\delta \{J^*(v_0^*,z_0^*)-\langle (u_0)_\alpha, (N_0)_{\alpha\beta,\beta}+P_\alpha \rangle_{L^2}-\langle w_0, (Q_0)_{\alpha,\alpha}+P\rangle_{L^2}\}=0.$$

 Under such hypotheses,
 \begin{eqnarray} J(u_0)&=& \min_{u \in U} J(u) \nonumber \\ &=& \max_{v^* \in A^*} \tilde{J}^*(v^*)
 \nonumber \\ &=& \tilde{J}^*(v_0^*) \nonumber \\ &=& J^*(v_0^*,z_0^*). \end{eqnarray}

\end{thm}
\begin{proof}

Observe that, from the general result in Toland \cite{12}, we have

\begin{eqnarray}
\inf_{z^* \in Y_2^*} J^*(v^*,z^*)&=& \inf_{z^* \in Y_2^*}\{ -F^*(z^*,Q)+G^*(z^*,N)\} \nonumber \\
 &=& \inf_{z^* \in Y_2^*}\left\{ -F^*(z^*,Q)+\frac{1}{2}\int_\Omega \overline{(-N_{\alpha\beta})^K}\;z^*_\alpha \;z^*_\beta\;dx \right.\nonumber \\ && \left.-
\frac{1}{2}\int_\Omega \overline{H}_{\alpha\beta\lambda\mu}N_{\alpha\beta}N_{\lambda \mu}\;dx\right\}
\nonumber \\  &\leq& -\langle Q_\alpha+z^*_\alpha \;,\;w_{,\alpha} \rangle_{L^2} +F(u)+\langle z^*_\alpha\;,\;w_{,\alpha} \rangle_{L^2}+\frac{1}{2}\langle N_{\alpha\beta} -K \delta_{\alpha\beta}, w_{,\alpha} w_{,\beta}
 \rangle_{L^2}\nonumber \\ && -
\frac{1}{2}\int_\Omega \overline{H}_{\alpha\beta\lambda\mu}N_{\alpha\beta}N_{\lambda \mu}\;dx\nonumber \\ &=& -\langle Q_\alpha , w_{,\alpha} \rangle_{L^2} +F(u)-\frac{1}{2}\langle N_{\alpha\beta} -K \delta_{\alpha\beta}, w_{,\alpha} w_{,\beta}
 \rangle_{L^2}\nonumber \\ &&-
\frac{1}{2}\int_\Omega \overline{H}_{\alpha\beta\lambda\mu}N_{\alpha\beta}N_{\lambda \mu}\;dx.
\end{eqnarray}

From this,
\begin{eqnarray}
\inf_{z^* \in Y_2^*} J^*(v^*,z^*) &\leq& -\langle P , w \rangle_{L^2} +G_2(\kappa(u))+ \sup_{N \in Y_1^*}\left\{\frac{1}{2}\langle N_{\alpha\beta}, w_{,\alpha} w_{,\beta}
 \rangle_{L^2}\right.\nonumber \\ &&-
\left.\frac{1}{2}\int_\Omega \overline{H}_{\alpha\beta\lambda\mu}N_{\alpha\beta}N_{\lambda \mu}\;dx -\langle u_\alpha, N_{\alpha\beta,\beta}+P_\alpha
\rangle_{L^2}\right\}\nonumber \\ &=&-\langle P , w \rangle_{L^2} +G_2(\kappa(u))+ \sup_{N \in Y_1^*}\left\{\frac{1}{2}\langle N_{\alpha\beta}, u_{\alpha,\beta}+u_{\beta,\alpha}+ w_{,\alpha} w_{,\beta}
 \rangle_{L^2}\right.\nonumber \\ &&-
\left.\frac{1}{2}\int_\Omega \overline{H}_{\alpha\beta\lambda\mu}N_{\alpha\beta}N_{\lambda \mu}\;dx -\langle u_\alpha, P_\alpha
\rangle_{L^2}\right\}\nonumber \\ &=& G_2(\kappa(u))+G_1(\gamma(u))-\langle  w ,P\rangle_{L^2}-\langle u_\alpha, P_\alpha \rangle_{L^2}
\nonumber \\ &=& J(u),\;\; \forall u \in U,\; v^*=(Q,N) \in A^*.
\end{eqnarray}

Thus,
\begin{eqnarray}
J(u) &\geq& \inf_{z^* \in Y_2^*}\{- F^*(z^*,Q)+G^*(z^*,N)\} \nonumber \\ &=& \inf_{z^* \in Y_2^*} J^*(v^*,z^*)
\nonumber \\ &=& \tilde{J}^*(v^*), \end{eqnarray}
$\forall v^* \in A^*,\; u\in U.$

Summarizing,
$$J(u) \geq \tilde{J}^*(v^*),\; \forall u \in U,\; v^* \in A^*,$$
so that,

\begin{equation}\label{usp.55a}
\inf_{u \in U} J(u) \geq \sup_{v^* \in A^*} \tilde{J}^*(v^*).
\end{equation}
Finally, assume $v_0^*=(N_0,Q_0) \in A^*$, $z_0^* \in Y_2^*$ and $u_0 \in U$ are such that
$$\delta\{J^*(v_0^*,z_0^*)-\langle (u_0)_\alpha, (N_0)_{\alpha\beta,\beta}+P_\alpha \rangle_{L^2}-\langle w_0,(Q_0)_{\alpha,\alpha}+P\rangle_{L^2}\}=0.$$

From the variation in $Q^*$ we get
$$\frac{\partial F^*(z_0^*,Q_0)}{\partial Q_\alpha}
=(w_0)_{,\alpha},
$$
so that from this and the variation in $z^*$, we get
\begin{eqnarray}\label{ps.88a}\frac{\partial G^*(z_0^*,N_0)}{\partial z^*_\alpha} &=&
\frac{\partial F^*(z_0^*,Q_0)}{\partial z^*_\alpha}
\nonumber \\ &=& \frac{\partial F^*(z_0^*,Q_0)}{\partial Q_\alpha} \nonumber \\ &=& (w_0)_{,\alpha} \nonumber \\
&=& \overline{N_{\alpha\beta}^K}(z_0^*)_{\beta}, \text{ in } \Omega.\end{eqnarray}
Hence,
$$F^*(z_0^*,Q_0)  =\langle (z_0^*)_\alpha+(Q_0)_\alpha,(w_0)_{,\alpha} \rangle_{L^2}-F(u_0).$$

Also, from (\ref{ps.88a}) we have,
\begin{equation}\label{ps2a}(z_0^*)_\alpha=(N_0)_{\alpha\beta}(w_0){,_\beta}+K (w_0)_{,\alpha}.\end{equation}

From such results and the variation in $N$ we obtain
$$\frac{(u_0)_{\alpha,\beta}+(u_0)_{\beta,\alpha}}{2}+\frac{(w_0)_{,\alpha} (w_0)_{,\beta}}{2}-\overline{H}_{\alpha\beta\gamma\mu}(N_0)_{\lambda\mu}=0,$$
so that
$$(N_0)_{\alpha\beta}=H_{\alpha\beta\lambda\mu}\gamma_{\lambda\mu}(u_0).$$

From these last results we have,
\begin{eqnarray}
G^*(z_0^*,N_0)&=&\langle (z_0^*)_\alpha, (w_0)_\alpha \rangle_{L^2}+\frac{1}{2}\left\langle (N_0)_{\alpha\beta}-K\delta_{\alpha\beta},(w_0)_\alpha (w_0)_\beta \right\rangle_{L^2} \nonumber \\ &&-\frac{1}{2}\int_\Omega \overline{H_{\alpha\beta\lambda \mu}}(N_0)_{\alpha\beta}(N_0)_{\lambda \mu}\;dx
\nonumber \\ &=& \langle (z_0^*)_\alpha, (w_0)_\alpha \rangle_{L^2}+\frac{1}{2}\left\langle (N_0)_{\alpha\beta},\frac{(u_0)_{\alpha,\beta}+(u_0)_{\beta,\alpha}}{2} +(w_0)_\alpha (w_0)_\beta \right\rangle_{L^2} \nonumber \\ &&-\langle P_\alpha, (u_0)_\alpha\rangle_{L^2}-\frac{1}{2}\int_\Omega \overline{H_{\alpha\beta\lambda \mu}}(N_0)_{\alpha\beta}(N_0)_{\lambda \mu}\;dx
\nonumber \\ &&-\frac{K}{2}
\langle (w_0)_{,\alpha},(w_0)_{,\alpha} \rangle_{L^2} \nonumber \\ &=& \langle (z_0^*)_\alpha, (w_0)_\alpha \rangle_{L^2}+ G_1(\gamma(u_0))-\frac{K}{2}
\langle (w_0)_{,\alpha},(w_0)_{,\alpha} \rangle_{L^2}\nonumber \\ &&-\langle P_\alpha, (u_0)_\alpha\rangle_{L^2} .
\end{eqnarray}

Joining the pieces, we obtain
\begin{eqnarray}\label{usp55a}J^*(v_0,z_0^*)&=& -F^*(z_0^*,Q_0)+G^*(z_0^*,N_0)
\nonumber \\ &=& F(u_0)+G_1(\gamma(u_0))+\frac{K}{2}\langle (w_0)_\alpha, (w_0)_\alpha \rangle_{L^2} \nonumber \\ && -\langle P, w_0 \rangle_{L^2}-
 \langle P_\alpha, (u_0)_\alpha\rangle_{L^2}\nonumber \\ &=& G_1(\gamma(u_0))+G_2(\kappa(u_0)) \nonumber \\ && -\langle P, w_0 \rangle_{L^2}-
 \langle P_\alpha, (u_0)_\alpha\rangle_{L^2}  \nonumber \\ &=&J(u_0).\end{eqnarray}

From  this and from $v_0^* \in A_4$, we have

$$J(u_0)=J^*(v_0^*,z^*_0)=\inf_{z^* \in Y_2^*} J^*(v^*_0,z^*)=\tilde{J^*}(v_0^*).$$

Therefore, from such a last equality  and (\ref{usp.55a}), we may infer that
\begin{eqnarray} J(u_0)&=& \min_{u \in U} J(u) \nonumber \\ &=& \max_{v^* \in A^*} \tilde{J}^*(v^*)
 \nonumber \\ &=& \tilde{J}^*(v_0^*) \nonumber \\ &=& J^*(v_0^*,z_0^*). \end{eqnarray}

 The proof is complete.
 \end{proof}

\section{Existence and duality for a non-linear shell model} In this section we present the primal variational formulation concerning
the shell model presented in \cite{2900}. Indeed, in some sense, the duality principle here presented extends the
results developed in \cite{2900}. In fact, through a  generalization of some ideas developed in \cite{12}, we have obtained a
duality principle for which the membrane force tensor, concerning the global optimal solution of the primal formulation, may not be necessarily either positive or negative definite. We emphasize details on the Sobolev spaces involved may be found in
\cite{1}.  About the fundamental concepts of convex analysis and duality here used, we would cite \cite{29,[6],120}.
Similar problems are addressed in \cite{120,10,11,9}.

At this point we start to describe the shell model in question.
    Let $D\subset\mathbb{R}^2$ be an open, bounded, connected set  with a $C^3$ class boundary.

    Let $S \subset \mathbb{R}^3$ be a $C^3$ class manifold, where
    $$S=\{\mathbf{r}(\xi)\;:\; \xi=(\xi_1,\xi_2) \in \overline{D}\},$$
$\mathbf{r}:\overline{D} \subset \mathbb{R}^2 \rightarrow \mathbb{R}^3$ is a  $C^3$ class function and where,
$$\mathbf{r}(\xi)=X_1(\xi)\mathbf{e}_1+X_2(\xi) \mathbf{e}_2+X_3(\xi)\mathbf{e}_3.$$

Here, $\{\mathbf{e}_1,\mathbf{e}_2,\mathbf{e}_3\}$ is the canonical basis of $\mathbb{R}^3$.

We assume $S$ is the middle surface of a shell of constant thickness $h$ so that we denote,

$$\mathbf{a}_\alpha=\frac{\partial \mathbf{r}}{\partial \mathbf{\xi_\alpha}},\; \forall \alpha \in \{1,2\},$$
$$a_{\alpha\beta}=\mathbf{a}_\alpha \cdot \mathbf{a}_\beta.$$

Let $$\mathbf{n}=\frac{\mathbf{a}_\alpha \times \mathbf{a}_\beta}{|\mathbf{a}_\alpha \times \mathbf{a}_\beta|},$$
be the unit normal to $S$, so that we define the covariant components $b_{\alpha\beta}$ of the
curvature tensor $$b=\{b_{\alpha\beta}\},$$ by
$$b_{\alpha\beta}=\mathbf{n}\cdot \mathbf{a}_{\alpha, \beta}=\mathbf{n}\cdot \mathbf{r}_{,\alpha\beta}.$$

Observe that $$\mathbf{n}\cdot \mathbf{a}_\alpha=0,$$ so that
$$\mathbf{n}_\beta \cdot \mathbf{a}_\alpha+ \mathbf{n}\cdot \mathbf{a}_{\alpha, \beta}=0,$$
and thus we obtain,
$$b_{\alpha\beta}=\mathbf{n}\cdot \mathbf{a}_{\alpha, \beta}=-\mathbf{n}_\beta \cdot \mathbf{a}_\alpha.$$

The Christofell symbols relating $S$, would be,
$$\Gamma_{\alpha\beta\gamma}=\frac{1}{2}(a_{\alpha\beta, \gamma}+a_{\alpha\gamma,\beta}-a_{\beta\gamma,\alpha}),\; \forall \alpha,\beta,\gamma \in \{1,2\}$$
and
$$\Gamma_{\alpha\beta}^\lambda=a^{\lambda \gamma}\Gamma_{ \gamma\alpha\beta},\; \forall \alpha,\beta,\lambda \in \{1,2\},$$
where
$$\{a^{\alpha\beta}\}=\{a_{\alpha\beta}\}^{-1}.$$

Let us denote with a bar the quantities relating the deformed middle surface.

So, the middle surface strain tensor $\gamma=\{\gamma_{\alpha\beta}\}$ is given by
$$\gamma_{\alpha\beta}=\frac{(\bar{a}_{\alpha\beta}-a_{\alpha\beta})}{2},$$
while the tensor relating change in curvature is given by
$$\kappa_{\alpha\beta}=-(\bar{b}_{\alpha\beta}-b_{\alpha\beta}), \; \forall \alpha,\beta \in \{1,2\}.$$

We also denote,

\begin{eqnarray} U&=&\{u=(u_\alpha,w)=(u_1,u_2,w) \in W^{1,2}(S; \mathbb{R}^2) \times W^{2,2}(S)\;:
\nonumber \\ && u_\alpha=w= \frac{\partial w}{\partial \mathbf{n}}=0, \text{ on } \partial S\} \nonumber \\ &=&
W_0^{1,2}(S;\mathbb{R}^2) \times W_0^{2,2}(S),
\end{eqnarray}
where, in order to simplify the analysis, the boundary conditions in question refer to a clamped shell and where
$\partial S$ denotes the boundary of $S$.

Also from reference \cite{2900}, for moderately large rotations around tangents, the strain displacements
relations are given by,

\begin{enumerate}
\item $$\gamma_{\alpha\beta}(u)=\theta_{\alpha\beta}(u)+\frac{1}{2}\varphi_\alpha(u) \varphi_\beta(u),$$
\item \begin{eqnarray}\kappa_{\alpha\beta}(u)&=&-w_{|\alpha\beta}-b_{\alpha | \beta}^\lambda u_\lambda
 \nonumber \\ &&-b_{\alpha}^\lambda u_{\lambda | \beta}-b_{\beta}^\lambda u_{\lambda | \alpha}+b_\alpha^\lambda b_{\lambda \beta} w, \end{eqnarray}
\end{enumerate}
where,

$$u_{\alpha| \beta}=u_{\alpha,\beta}-\Gamma_{\alpha\beta}^\lambda u_\lambda,$$
$$w_{|\alpha\beta}=w_{,\alpha\beta}-\Gamma_{\alpha\beta}^\lambda w_{,\lambda},$$
$$\theta_{\alpha\beta}(u)=\frac{1}{2}(u_{\alpha|\beta}+u_{\beta|\alpha})-b_{\alpha\beta}w,$$
\begin{equation}\label{b.1}\varphi_\alpha(u)=w_{,\alpha}+b_\alpha^\beta u_\beta,\end{equation}
and
$$b_\alpha^\beta=b_{\alpha \lambda} a^{\lambda \beta}.$$

The primal shell inner energy is defined by
\begin{eqnarray}W(\gamma(u),\kappa(u))&=&\frac{1}{2}\int_SH^{\alpha\beta\lambda\mu}\gamma_{\alpha\beta}(u)\gamma_{\lambda \mu}(u)\;dS \nonumber \\ &&+\frac{1}{2}\int_Sh^{\alpha\beta\lambda\mu}\kappa_{\alpha\beta}(u)\kappa_{\lambda \mu}(u)\;dS,
\end{eqnarray}
where
$$H^{\alpha\beta\lambda\mu}=\frac{Eh}{2(1+\nu)}\left(a^{\alpha\lambda}a^{\beta \mu}+a^{\alpha\mu}a^{\beta\lambda}+\frac{2\nu}{1-\nu}a^{\alpha\beta}a^{\lambda \mu}\right),$$

$$h^{\alpha\beta\lambda\mu}=\frac{h^2}{12}H^{\alpha\beta\lambda\mu},$$

and $E$ denotes the Young's modulus and $\nu$ the Poisson ratio.

The constitutive relations are,
$$N^{\alpha\beta}=H^{\alpha\beta\lambda\mu}\gamma_{\lambda\mu}(u),$$
and
$$M^{\alpha\beta}=h^{\alpha\beta\lambda\mu}\kappa_{\lambda\mu}(u),$$

where $\{N^{\alpha\beta}\}$ is the membrane force tensor and $\{M^{\alpha\beta}\}$ is the moment one.

We assume $$H=\{H^{\alpha\beta\lambda\mu}\}$$ to be positive definite in the sense that there exists $c_0>0$
such that
$$H^{\alpha\beta\lambda\mu}t_{\alpha\beta}t_{\lambda\mu} \geq c_0 t_{\alpha\beta}t_{\alpha\beta} \geq 0, \; \forall t=\{t_{\alpha\beta}\} \in \mathbb{R}^4.$$

Finally, the primal variational formulation for this model  will be given by
$$J:U \rightarrow  \mathbb{R}$$ where,
$$J(u)=W(\gamma(u),\kappa(u))-\langle u,\mathbf{f} \rangle_{L^2},$$
$$W(\gamma(u),\kappa(u))=G_1(\gamma(u))+G_2(\kappa(u)),$$
$$G_1(\gamma(u))=\frac{1}{2}\int_SH^{\alpha\beta\lambda\mu}\gamma_{\alpha\beta}(u)\gamma_{\lambda \mu}(u)\;dS,$$
$$G_2(\kappa(u))=\frac{1}{2}\int_Sh^{\alpha\beta\lambda\mu}\kappa_{\alpha\beta}(u)\kappa_{\lambda \mu}(u)\;dS,$$
and $$\langle u,\mathbf{f} \rangle_{L^2}=\int_S(P^\alpha u_\alpha+Pw)\;dS.$$

Here $$\mathbf{f}=(P^\alpha,P) \in L^2(S;\mathbb{R}^3),$$
are the external loads distributed on $S$, $P^\alpha$ relating the directions $\mathbf{a}_\alpha$ and $P$ relating the direction $\mathbf{n}$, respectively.

Moreover,  generically for $f_1,f_2 \in L^2(S),$ we denote,
$$ \langle f_1, f_2 \rangle_{L^2}=\int_S f_1 f_2\;dS,$$
where $dS= \sqrt{a} \;d\xi_1\;d\xi_2,$ and $a=det\{a_{\alpha \beta}\}.$
\section{Existence of a minimizer}
\begin{thm}\label{us123b}Consider the statements and assumptions concerning the shell model described in the last section.

More specifically, consider the functional $J:U \rightarrow \mathbb{R}$ given, as above described, by

$$J(u)=W(\gamma(u),\kappa(u))-\langle u,\mathbf{f} \rangle_{L^2},$$
where,
 \begin{eqnarray} U&=&\{u=(u_\alpha,w)=(u_1,u_2,w) \in W^{1,2}(S; \mathbb{R}^2) \times W^{2,2}(S)\;:
\nonumber \\ && u_\alpha=w= \frac{\partial w}{\partial \mathbf{n}}=0, \text{ on } \partial S\} \nonumber \\
&=& W_0^{1,2}(S;\mathbb{R}^2) \times W^{2,2}_0(S).
\end{eqnarray}
Let $T_0$ be such that,
$$\|T_0\|_2^2=\min_{ T \in L^2(S;\mathbb{R}^{2 \times 2})}\{\|T\|_2^2\},$$
subject to $$\frac{(\sqrt{g} T_{\alpha\beta})_{,\beta}}{\sqrt{g}}+\Gamma^\alpha_{\lambda\beta}T_{\lambda\beta}+P_\alpha=0 \text{ in }D.$$

Assume \begin{equation}\label{e.2}J_1(u) \rightarrow +\infty, \text{ as } \langle w_{,\alpha\beta},w_{,\alpha\beta} \rangle_{L^2} +\langle u_{\alpha,\beta}
,u_{\alpha,\beta} \rangle_{L^2} \rightarrow +\infty,\end{equation}
where
$$J_1(u)=G_2(\kappa(u))+\frac{1}{2}\langle (T_0)_{\alpha\beta}, \varphi_\alpha(u) \varphi_\beta(u) \rangle -\langle (T_0)_{\alpha \beta}
b_{\alpha\beta}+P,w \rangle_{L^2}.$$

Under such hypotheses, there exists $u_0 \in U$ such that,
$$J(u_0)=\min_{u \in U} \{J(u)\}.$$
\end{thm}

\begin{proof}

Observe that
\begin{eqnarray}\langle P_\alpha,u_\alpha \rangle_{L^2}&=&-\int_D\left(\frac{(\sqrt{g} (T_0)_{\alpha\beta})_{,\beta}}{\sqrt{g}}+\Gamma^\alpha_{\lambda\beta}(T_0)_{\lambda\beta}\right)u_\alpha \sqrt{g}\;d\xi
\nonumber \\ &=& \left\langle (T_0)_{\alpha\beta}, \frac{u_{\alpha,\beta}+u_{\beta,\alpha}}{2}-\Gamma^\lambda_{\alpha\beta}u_\lambda \right\rangle_{L^2}
\nonumber \\ &=& \langle (T_0)_{\alpha\beta},\theta_{\alpha\beta}(u)+b_{\alpha\beta}w \rangle_{L^2}, \end{eqnarray}

so that defining \begin{equation}\label{b.2}v_{\alpha\beta}(u)=\theta_{\alpha\beta}(u)+\frac{1}{2}\varphi_\alpha(u)\varphi_\beta(u),\end{equation}
we obtain
$$\langle P_\alpha,u_\alpha \rangle_{L^2}=\left\langle (T_0)_{\alpha\beta},v_{\alpha\beta}(u)-\frac{1}{2}\varphi_\alpha(u)\varphi_\beta(u)+b_{\alpha\beta}w \right\rangle_{L^2}.$$

Thus,
\begin{eqnarray}\label{usp9}
J(u)&=& G_1(v(u))-\langle (T_0)_{\alpha\beta},v_{\alpha\beta}(u) \rangle_{L^2}
\nonumber \\ &&+G_2(\kappa (u))+\frac{1}{2}\langle (T_0)_{\alpha\beta}, \varphi_\alpha(u) \varphi_\alpha(u) \rangle_{L^2}
 \nonumber \\ &&-\langle (T_0)_{\alpha \beta}
b_{\alpha\beta}+P,w \rangle_{L^2},
\end{eqnarray}

From this and hypothesis (\ref{e.2}), $J$ is bounded below. So, there exists $\alpha_1 \in \mathbb{R}$ such that
$$\alpha_1=\inf_{u \in U} J(u).$$

Let $\{u_n\}$ be a minimizing sequence for $J$.

From (\ref{usp9}) and also from the hypotheses (\ref{e.2}), $\{\|w_n\|_{2,2}\}$ and $\{\|(u_n)_\alpha\|_{1,2}\}$ are bounded.

So there exists $w_0 \in W^{2,2}_0(S)$ and $\{(u_0)_\alpha\} \in W^{1,2}_0(S;\mathbb{R}^2)$ such that, up to a not relabeled subsequence,
$$(w_n)_{,\alpha\beta} \rightharpoonup (w_0)_{,\alpha\beta}
\text{ weakly in } L^2(S),\; \forall \alpha,\beta \in \{1,2\},$$
\begin{equation}\label{b.3}(w_n)_{,\alpha} \rightarrow (w_0)_{,\alpha}
\text{ strongly in } L^2(S),\; \forall \alpha\in \{1,2\},\end{equation}

\begin{equation}\label{b.5}(u_n)_{\alpha,\beta} \rightharpoonup (u_0)_{\alpha,\beta}
\text{ weakly in } L^2(S),\; \forall \alpha,\beta \in \{1,2\},\end{equation}
\begin{equation}\label{b.4}(u_n)_{\alpha} \rightarrow (u_0)_{\alpha}
\text{ strongly in } L^2(S),\; \forall \alpha\in \{1,2\}.\end{equation}

Also from (\ref{usp9}), there exists $K_2>0$ such that,
$$\|(v_n)_{\alpha\beta}(u)\|_2< K_2,\forall \alpha,\beta \in \{1,2\},\;n \in \mathbb{N},$$
and thus, from this, (\ref{b.1}), (\ref{b.2}), (\ref{b.3}) and (\ref{b.4}), we may infer that
there exists $K_3>0$ such that
$$\|\theta_{\alpha \beta}(u_n)\|_2< K_3,\forall \alpha,\beta \in \{1,2\},\;n \in \mathbb{N}.$$

 Thus, from this, (\ref{b.4}) and (\ref{b.5}), up to a subsequence not relabeled,
 $$\theta_{\alpha \beta}(u_n) \rightharpoonup \theta_{\alpha\beta}(u_0),\; \text{ weakly in } L^2,$$
 $\forall \alpha,\beta \in \{1,2\}, \text{ as } n \rightarrow \infty.$

From this, also up to a subsequence not relabeled, we have
$$\gamma_{\alpha\beta}(u_n) \rightharpoonup \gamma_{\alpha\beta}(u_0), \text{ weakly in } L^2,$$
$\forall \alpha,\beta \in \{1,2\},$
and
$$\kappa_{\alpha\beta}(u_n) \rightharpoonup \kappa_{\alpha\beta}(u_0), \text{ weakly in } L^2,$$
$\forall \alpha,\beta \in \{1,2\}.$

Therefore, from the convexity of $G_1$ in $\gamma$ and $G_2$ in $\kappa$ we obtain
\begin{eqnarray}
\inf_{u \in U}J(u)&=& \alpha_1 \nonumber \\ &=& \liminf_{n \rightarrow \infty}J(u_n) \nonumber \\ &\geq&
J(u_0).
\end{eqnarray}

Thus, $$J(u_0)=\min_{u\in U}J(u).$$

The proof is complete.

\end{proof}

\section{The duality principle} At this point we present the main duality principle, which is summarized by the next
theorem.

\begin{thm} Consider the statements and assumptions concerning the shell model described in the last two sections.

More specifically, consider the functional $J:U \rightarrow \mathbb{R}$ given, as above described by,

$$J(u)=W(\gamma(u),\kappa(u))-\langle u,\mathbf{f} \rangle_{L^2},$$
where,
 \begin{eqnarray} U&=&\{u=(u_\alpha,w)=(u_1,u_2,w) \in W^{1,2}(S; \mathbb{R}^2) \times W^{2,2}(S)\;:
\nonumber \\ && u_\alpha=w= \frac{\partial w}{\partial \mathbf{n}}=0, \text{ on } \partial S\} \nonumber \\
&=& W_0^{1,2}(S;\mathbb{R}^2) \times W^{2,2}_0(S).
\end{eqnarray}
Under such hypotheses,
\begin{eqnarray}&& \inf_{u \in U} J(u) \nonumber \\ &\geq&
\sup_{v^* \in A^*}\{\inf_{z^* \in Y_2^*}\{-F^*(z^*,Q)+G^*(z^*,N)\}\},
\end{eqnarray}
where,
$$F(u)=G_2(\kappa(u))+\frac{K}{2} \langle \varphi_{\alpha}(u),\varphi_{\alpha}(u) \rangle_{L^2}, \forall u \in U.$$

Moreover, $F^*:[Y_2^*]^2 \rightarrow \mathbb{R}$ is defined by,
$$F^*(z^*,Q)=\sup_{u \in U}\{ \langle z^*_{\alpha}+Q_\alpha,\varphi_{\alpha}(u) \rangle_{L^2}-F(u)\},\; \forall (z^*,Q) \in [Y_2^*]^2.$$

Also,

\begin{eqnarray} G(v)&=&-\frac{1}{2}\int_S H_{\alpha\beta\lambda\mu}\left[(v_1)_{\alpha\beta}+\frac{(v_2)_\alpha (v_2)_\beta}{2}\right]
\left[(v_1)_{\lambda\mu}+\frac{(v_2)_\lambda (v_2)_\mu}{2}\right]\;dS \nonumber \\ &&+\frac{K}{2}\langle (v_2)_\alpha,(v_2)_\alpha \rangle_{L^2},
\end{eqnarray}
\begin{eqnarray}
G^*(z^*,N)&=&
\sup_{v_2 \in Y_2}\{\inf_{v_1 \in Y_1} \{\langle N^{\alpha\beta}, (v_1)_{\alpha\beta} \rangle_{L^2}+\langle z^*_\alpha , (v_2)_\alpha \rangle_{L^2}-G(v)\}\}
\nonumber \\ &=& \frac{1}{2}\int_S \overline{(-N^{\alpha\beta})^K}z^*_\alpha z^*_\beta\;dS \nonumber \\ &&-
\frac{1}{2}\int_S \overline{H}_{\alpha\beta\lambda\mu}N^{\alpha\beta}N^{\lambda \mu}\;dS,\end{eqnarray}
if $v^*=(Q,N) \in A_3,$ where
$$A_3=\{v^* \in Y^*\;:\;\{(-N^{\alpha\beta})_K\}\text{ is positive definite in } \Omega\},$$

 \begin{equation} \{(-N^{\alpha\beta})_K\}= \left\{
\begin{array}{lr}
 -N^{11}+K  &  -N^{12}
 \\
 -N^{21} & -N^{22}+K
 \end{array} \right\},\end{equation}
and
 $$\overline{(-N^{\alpha\beta})_K}=\{(-N^{\alpha\beta})_K\}^{-1}.$$

Moreover, defining
$$Y^*=Y=L^2(S; \mathbb{R}^2) \times L^2(S;\mathbb{R}^{2 \times 2}),$$
$$Y_1^*=Y_1=L^2(S;\mathbb{R}^{2\times 2}),$$ and $$Y_2^*=Y_2=L^2(S;\mathbb{R}^2),$$

also, $$A^*=A_1 \cap A_2\cap A_3 \cap A_4$$ where,
\begin{eqnarray}
A_1&=&\{v^*=(Q,N) \in Y^*\;:\;-N^{\alpha\beta}|_\beta+b_\lambda^\alpha Q^\lambda
-P^\alpha=0, \text{ in } S\}, \end{eqnarray}

\begin{eqnarray}
A_2&=&\{v^*=(Q,N) \in Y^*\;:\; -b_{\alpha\beta}N^{\alpha\beta}-Q^\alpha_{|\alpha}-P=0, \text{ in } S\}.
\end{eqnarray}
 Moreover,
 $$\{\overline{H}_{\alpha\beta\lambda\mu}\}=\{H^{\alpha\beta\lambda\mu}\}^{-1},$$
 and,
$$\{\overline{h}_{\alpha\beta\lambda\mu}\}=\{h^{\alpha\beta\lambda\mu}\}^{-1}.$$
Also,
\begin{eqnarray}A_4&=&\{v^*=(Q,N) \in Y^*\;:\; \hat{J}^*(z^*)>0, \nonumber \\ &&\; \forall z^* \in Y_2^*, \; \text{ such that } z^* \neq \mathbf{0}\},
\end{eqnarray}
where, \begin{eqnarray}\hat{J}^*(z^*)&=& -F^*(z^*,\mathbf{0})+G^*(z^*,\mathbf{0}) \nonumber \\ &=& -F^*(z^*,\mathbf{0})+\frac{1}{2}\int_S \overline{(-N^{\alpha\beta})_K} z^*_\alpha z^*_\beta\;dS,\; \forall z^* \in Y_2^*.\nonumber\end{eqnarray}
 Finally, denoting
 $$J^*(v^*,z^*)=-F^*(z^*,Q)+G^*(z^*,N), \forall (v^*,z^*) \in A^* \times Y_2^*,$$
 and
 $$\tilde{J}^*(v^*)=\inf_{z^* \in Y_2^*} J^*(v^*,z^*), \forall v^* \in A^*,$$
 suppose there exist $v_0^*=(N_0,Q_0) \in A^*$, $z_0^* \in Y_2^*$ and $u_0 \in U$ such that
 \begin{eqnarray}&&\delta \{J^*(v_0^*,z_0^*)-\langle \theta_{\alpha \beta}(u_0),(N_0)^{\alpha\beta} \rangle_{L^2}
 \nonumber \\ &&+\langle (u_0)_\alpha, P_\alpha \rangle_{L^2}-\langle (\varphi_0)_\alpha, (Q_0)_{\alpha}
 \rangle_{L^2}+\langle w_0, P\rangle_{L^2}\}=0.\end{eqnarray}

 Under such hypotheses,
 \begin{eqnarray} J(u_0)&=& \min_{u \in U} J(u) \nonumber \\ &=& \max_{v^* \in A^*} \tilde{J}^*(v^*)
 \nonumber \\ &=& \tilde{J}^*(v_0^*) \nonumber \\ &=& J^*(v_0^*,z_0^*). \end{eqnarray}

\end{thm}
\begin{proof}

Observe that, from the general result in Toland \cite{12}, we have

\begin{eqnarray}
\inf_{z^* \in Y_2^*} J^*(v^*,z^*)&=& \inf_{z^* \in Y_2^*}\{ -F^*(z^*,Q)+G^*(z^*,N)\} \nonumber \\
 &=& \inf_{z^* \in Y_2^*}\left\{ -F^*(z^*,Q)+\frac{1}{2}\int_S \overline{(-N^{\alpha\beta})_K}z^*_\alpha z^*_\beta\;;dx \right.\nonumber \\ && \left.-
\frac{1}{2}\int_S \overline{H}_{\alpha\beta\lambda\mu}N^{\alpha\beta}N^{\lambda \mu}\;dS\right\}
\nonumber \\  &\leq& -\langle Q_\alpha+z^*_\alpha ;,\;\varphi_\alpha(u) \rangle_{L^2} +F(u) \nonumber \\ &&+\langle z^*_\alpha\;,\;\varphi_\alpha(u) \rangle_{L^2}+\frac{1}{2}\langle N^{\alpha\beta} -K \delta_{\alpha\beta}, \varphi_\alpha(u) \varphi_\beta(u)
 \rangle_{L^2}\nonumber \\ && -
\frac{1}{2}\int_S \overline{H}_{\alpha\beta\lambda\mu}N^{\alpha\beta}N^{\lambda \mu}\;dS\nonumber \\ &=& -\langle Q_\alpha , \varphi_\alpha(u) \rangle_{L^2} +F(u)-\frac{1}{2}\langle N^{\alpha\beta} -K \delta_{\alpha\beta}, \varphi_\alpha(u) \varphi_\beta(u)
 \rangle_{L^2}\nonumber \\ &&-
\frac{1}{2}\int_\Omega \overline{H}_{\alpha\beta\lambda\mu}N^{\alpha\beta}N^{\lambda \mu}\;dS
\end{eqnarray}

From this, we have
\begin{eqnarray}
\inf_{z^* \in Y_2^*} J^*(v^*,z^*)
 &\leq& -\langle P , w \rangle_{L^2} +G_2(\kappa(u))+ \sup_{N \in Y_1^*}\left\{\frac{1}{2}\langle N^{\alpha\beta}, \varphi_\alpha(u) \varphi_\beta(u)
 \rangle_{L^2}\right.\nonumber \\ &&-
\left.\frac{1}{2}\int_\Omega \overline{H}_{\alpha\beta\lambda\mu}N^{\alpha\beta}N^{\lambda \mu}\;dS -\langle \theta_{\alpha \beta}(u), N^{\alpha\beta}
\rangle_{L^2}-\langle u_\alpha, P_\alpha \rangle_{L^2}
\rangle_{L^2}\right\}\nonumber \\ &=&-\langle P , w \rangle_{L^2} +G_2(\kappa(u))+ \sup_{N \in Y_1^*}\left\{\frac{1}{2}\langle N^{\alpha\beta}, \theta_{\alpha\beta}(u)
+ \varphi_{\alpha}(u) \varphi_\beta (u)
 \rangle_{L^2}\right.\nonumber \\ &&-
\left.\frac{1}{2}\int_S \overline{H}_{\alpha\beta\lambda\mu}N^{\alpha\beta}N^{\lambda \mu}\;dS -\langle u_\alpha, P_\alpha
\rangle_{L^2}\right\}\nonumber \\ &=& G(\kappa(u))+G_1(\gamma(u))-\langle  w ,P\rangle_{L^2}-\langle u_\alpha, P_\alpha\rangle_{L^2}
\nonumber \\ &=& J(u), \forall u \in U,\; v^*=(Q,N) \in A^*.
\end{eqnarray}

Thus,
\begin{eqnarray}
J(u) &\geq& \inf_{z^* \in Y_2^*}\{ -F^*(z^*,Q)+G^*(z^*,N)\} \nonumber \\ &=& \inf_{z^* \in Y_2^*} J^*(v^*,z^*)
\nonumber \\ &=& \tilde{J}^*(v^*), \end{eqnarray}
$\forall v^* \in A^*,\; u\in U.$

Summarizing,
$$J(u) \geq \tilde{J}^*(v^*),\; \forall u \in U,\; v^* \in A^*,$$
so that,
\begin{equation}\label{usp.55}
\inf_{u \in U} J(u) \geq \sup_{v^* \in A^*} \tilde{J}^*(v^*).
\end{equation}
Finally, assume $v_0^*=(N_0,Q_0) \in A^*$, $z_0^* \in Y_2^*$ and $u_0 \in U$ are such that
\begin{eqnarray}&&\delta \{J^*(v_0^*,z_0^*)-\langle \theta_{\alpha \beta}(u_0),(N_0)^{\alpha\beta} \rangle_{L^2}
 \nonumber \\ &&+\langle (u_0)_\alpha, P_\alpha \rangle_{L^2}-\langle (\varphi_0)_\alpha, (Q_0)_{\alpha}
 \rangle_{L^2}+\langle w_0, P\rangle_{L^2}\}=0.\end{eqnarray}

From the variation in $Q^*$ we get
\begin{eqnarray}\frac{\partial F^*(z_0^*,Q_0)}{\partial Q_\alpha}
= \varphi_\alpha (u_0),
\end{eqnarray}
so that from this and the variation in $z^*$, we get
\begin{eqnarray}\label{ps.88}\frac{\partial G^*(z_0^*,Q)}{\partial z^*_\alpha}  &=&
\frac{\partial F^*(z_0^*,Q_0)}{\partial z^*_\alpha}
\nonumber \\ &=& \frac{\partial F^*(z_0^*,Q_0)}{\partial Q_\alpha} \nonumber \\ &=& \varphi_\alpha (u_0) \nonumber \\ &=& \overline{N^{\alpha\beta}_K}(z_0^*)_{\beta}, \text{ in } \Omega.\end{eqnarray}

Hence,
\begin{eqnarray}F^*(z_0^*,Q_0) &=&\langle (z_0^*)_\alpha+(Q_0)_\alpha,\varphi_\alpha (u_0) \rangle_{L^2}-F(u_0).\end{eqnarray}

Also, from (\ref{ps.88}) we have,
\begin{equation}\label{ps2}(z_0^*)_\alpha=(N_0)^{\alpha\beta}\varphi_\beta (u_0)+K \varphi_\alpha (u_0).\end{equation}

From such results and the variation in $N$ we obtain
$$\frac{(u_0)_{\alpha,\beta}+(u_0)_{\beta,\alpha}}{2}+\frac{\varphi_\alpha (u_0)\varphi_\beta (u_0)}{2}-\overline{H}_{\alpha\beta\gamma\mu}(N_0)^{\lambda\mu}=0,$$
so that
$$(N_0)^{\alpha\beta}=H_{\alpha\beta\lambda\mu}\gamma_{\lambda\mu}(u_0).$$

From these last results we obtain,
\begin{eqnarray}
G^*(z_0^*,N_0)&=&\langle (z_0^*)_\alpha, \varphi_\alpha (u_0)_\alpha \rangle_{L^2}+\frac{1}{2}\left\langle (N_0)^{\alpha\beta}-K\delta_{\alpha\beta},\varphi_\alpha (u_0)\varphi_\beta (u_0) \right\rangle_{L^2}\nonumber \\ &&-\frac{1}{2}\int_\Omega \overline{H_{\alpha\beta\lambda \mu}}(N_0)^{\alpha\beta}(N_0)^{\lambda \mu}\;dS \nonumber \\ &=& \langle (z_0^*)_\alpha, \varphi_\alpha (u_0)_\alpha \rangle_{L^2}+\left\langle (N_0)^{\alpha\beta},\theta_{\alpha\beta}(u_0) +\frac{1}{2}\varphi_\alpha (u_0)\varphi_\beta (u_0) \right\rangle_{L^2}\nonumber \\ &&-\langle P_\alpha, (u_0)_\alpha\rangle_{L^2}-\frac{1}{2}\int_\Omega \overline{H_{\alpha\beta\lambda \mu}}(N_0)^{\alpha\beta}(N_0)^{\lambda \mu}\;dS-\frac{K}{2}
\langle \varphi_\alpha (u_0),\varphi_\alpha (u_0) \rangle_{L^2} \nonumber \\ &=& \langle (z_0^*)_\alpha, \varphi_\alpha (u_0) \rangle_{L^2}+ G_1(\gamma(u_0))-\frac{K}{2}
\langle \varphi_\alpha (u_0),\varphi_\alpha (u_0) \rangle_{L^2} \nonumber \\ &&-\langle P_\alpha, (u_0)_\alpha\rangle_{L^2} .
\end{eqnarray}

Joining the pieces, we obtain
\begin{eqnarray}\label{usp5}J^*(v_0,z_0^*)&=& -F^*(z_0^*,Q_0)+G^*(z_0^*,N_0)
\nonumber \\ &=& F(u_0)+G_1(\gamma(u_0))+\frac{K}{2}\langle \varphi_\alpha (u_0), \varphi_\alpha (u_0) \rangle_{L^2} \nonumber \\ && -\langle P, w_0 \rangle_{L^2}-
 \langle P_\alpha, (u_0)_\alpha\rangle_{L^2}\nonumber \\ &=& G_1(\gamma(u_0))+G_2(\kappa(u_0)) \nonumber \\ && -\langle P, w_0 \rangle_{L^2}-
 \langle P_\alpha, (u_0)_\alpha\rangle_{L^2}  \nonumber \\ &=&J(u_0).\end{eqnarray}

From  this and from $v_0^* \in A_4$, we obtain

$$J(u_0)=J^*(v_0^*,z^*_0)=\inf_{z^* \in Y_2^*} J^*(v^*_0,z^*)=\tilde{J^*}(v_0^*).$$

Therefore, from such a last equality  and (\ref{usp.55}), we may infer that
\begin{eqnarray} J(u_0)&=& \min_{u \in U} J(u) \nonumber \\ &=& \max_{v^* \in A^*} \tilde{J}^*(v^*)
 \nonumber \\ &=& \tilde{J}^*(v_0^*) \nonumber \\ &=& J^*(v_0^*,z_0^*). \end{eqnarray}

 The proof is complete.
 \end{proof}

\section{Conclusion}
In this article, in a first step, we have developed new proofs of global existence of minimizers for the Kirchhoff-Love plate and a shell model
presented in \cite{2900}.

 In a second step, we have developed  duality principles for these same models. In \cite{2900}, the authors developed a duality principle valid for the special case in which the
membrane force tensor at a critical point is positive definite. We have generalized such a result, considering
that in our approach, such a previous case is included but here we do not request the optimal membrane force to be either positive or negative definite. Thus, in some sense, we have complemented the important work developed in \cite{2900}.  We would emphasize, sufficient optimality conditions are presented and the results here developed are applicable to a great variety of problems, including other shell models.


\end{document}